\documentclass[11pt]{amsart} 

\usepackage[utf8]{inputenc} 

\newtheorem{theorem}{Theorem}[section]
\newtheorem{corollary}[theorem]{Corollary}

\newtheorem{lemma}[theorem]{Lemma}

\theoremstyle{definition}
\newtheorem{definition}[theorem]{Definition}

\title{A characterization of the Filippov convention}
\author{Tomoharu Suda}
\address{Faculty of Mathematics, Keio University}
\email{tomoharu.suda@keio.jp}
\begin{document}
\begin{abstract}
The Filippov convention is widely used in the literature to define vector fields on a discontinuity set of piecewise-continuous vector fields. The aim of this paper is to give a sufficient and necessary condition for an interpolation scheme of piecewise-continuous vector fields to coincide with the Filippov convention. That is, we show that a map from a space of piecewise-continuous vector fields with two components to the space of vector fields coincides with the Filippov convention where the latter can be applied, if it is sufficiently well-behaved as a generalization of continuous vector fields. \end{abstract}

\maketitle
\section{Introduction}
 Piecewise-continuous vector fields arise in many applications and have attracted significant attention for several decades. Nontrivial phenomena such as slip-stick motion occur on a discontinuity set, and we are interested in studying the behavior of a system in a neighborhood of such points \cite{makarenkov2012dynamics}. Vector fields are, of course, not defined on the discontinuity set. One way to overcome this difficulty is to extend the definition of solutions. An overview of the studies along this line can be found in \cite{cortes2008discontinuous, polyakov2014stability}. Alternatively, one may consider an interpolation of a piecewise-continuous vector field and analyze the original system using the completed vector field. This approach has been taken in \cite{jeffrey2014hidden, jeffrey2018hidden}, for example. For considerations of this type, the problem of interpolation may be formulated as follows:

\textit{Let $PC(\mathbb{R}^n)$ be the space of all piecewise-continuous vector fields with two components. That is, $PC(\mathbb{R}^n)$ is the set of all triplets
$(X_1,X_2, \Sigma)$ where 
\begin{enumerate}
	\item $\Sigma$ is an $n-1$ dimensional smooth connected submanifold of $\mathbb{R}^n$, partitioning $\mathbb{R}^n$ into two nonempty regular closed sets $S_1$ and $S_2$ with disjoint interiors such that $S_1 \cap S_2 = \Sigma.$ The set $\Sigma$ is called the \emph{discontinuity manifold} of $(X_1,X_2, \Sigma).$
	\item $X_i$ is a continuous vector field defined on $S_i$ for $i=1,2.$
\end{enumerate}
 Construct a (partial) map $\eta$ from $PC(\mathbb{R}^n)$ to the space of possibly discontinuous vector fields $Map(\mathbb{R}^n, \mathbb{R}^n)$ so that $\eta(X_1,X_2, \Sigma)(x)= X_i(x)$ if $x \in \mathrm{int}\,S_i$ and $\eta(X , X, \Sigma)= X$ if $X$ is a continuous vector field.}

A solution $\eta$ to this problem will be called an \emph{interpolation scheme}. We permit an interpolation scheme to be invalid for some piecewise-continuous vector fields in $PC(\mathbb{R}^n),$ to accommodate for the possibility that the interpolation requires some regularity conditions.

One interpolation scheme widely used in the literature is Filippov's method based on convex combinations \cite{filippov, bernardo2008piecewise}. Although it was first formulated as a method to interpret piecewise-continuous vector fields as differential inclusions, it is now mainly used as a means of assigning a vector field to a discontinuity set. 
An interpolation scheme of the Filippov type can be defined as follows.
\begin{definition}
	An interpolation scheme $\eta$ is \emph{Filippov type} if 
	\begin{equation}\label{fil_def}
		\eta(X_1,X_2,\Sigma)(x) = \frac{X_{2N}(x)}{X_{2N}(x)-X_{1N}(x)} X_1(x)+ \frac{X_{1N}(x)}{X_{1N}(x)-X_{2N}(x)} X_2(x),
	\end{equation}
	for $x \in \Sigma$ with $X_{1N}(x) X_{2N}(x)<0.$ Here $X_{iN}$ is the normal component of $X_i$ defined by
	\[
		X_{iN}(x) := X_i(x) \cdot \mathbf{n}(x),
	\]
	where $ \mathbf{n}(x)$ is the unit normal vector of $\Sigma$ at $x.$ Note that the right-hand side of (\ref{fil_def}) does not depend on the choice of direction of $\mathbf{n}.$ 
	\end{definition}

The aim of this article is to show that a sufficiently well-behaved interpolation scheme of piecewise-continuous vector fields with two components is necessarily a Filippov type. 
Here we prepare some terms to state the main result.

A diffeomorphism $\Phi:\mathbb{R}^n\to \mathbb{R}^n$ defines a change of global coordinates. If $X$ is a vector field on $\mathbb{R}^n,$ it is transformed into $\Phi_* X := d\Phi \circ X \circ \Phi^{-1},$ where $d \Phi$ is the derivative of $\Phi.$ Therefore, a diffeomorphism $\Phi: \mathbb{R}^n \to \mathbb{R}^n$ induces a map $\Phi_*:PC(\mathbb{R}^n) \to PC(\mathbb{R}^n)$ by
\[
	\Phi_*(X_1, X_2, \Sigma) = (\Phi_* X_1, \Phi_* X_2, \Phi(\Sigma)).
\]
An interpolation scheme is \emph{invariant} if 
\[
	\eta\left(\Phi_*(X_1, X_2, \Sigma) \right) = \Phi_*\left(\eta(X_1, X_2, \Sigma)\right) 
\]
for all diffeomorphisms $\Phi:\mathbb{R}^n\to \mathbb{R}^n$ and all $(X_1,X_2, \Sigma) \in PC(\mathbb{R}^n).$ If an interpolation scheme is invariant, we may apply the usual rules of change of variables.

As the aim of interpolation is to obtain a globally defined vector field that respects the behavior of the original piecewise-continuous system, it is useful to require that the invariant sets of the original system be that of the interpolated vector field. An interpolation scheme $\eta$ satisfies the \emph{tangency condition} if $\eta(X_1,X_2,\Sigma)(x)$ is tangent to $\Sigma$ for $x \in \Sigma$ with $X_{1N}(x) X_{2N}(x)<0.$ In this case, the sliding invariant sets of the original system are invariant under the interpolated vector field.

For a chart $(U,\phi)$ of $x \in \Sigma$, each $(X_1,X_2, \Sigma)$ induces a map $\xi^U_{(X_1,X_2, \Sigma)}: \mathbb{R}^{n-1} \to \mathbb{R}^{2n}$ via 
\[
	\xi^{(U,\phi)}_{(X_1,X_2, \Sigma)}(z) = (X_1(\phi^{-1}z), X_2(\phi^{-1}z)).
\]
 We say an interpolation scheme to \emph{be locally uniform} if, for each chart $(U,\phi),$ there is a map $k_{(U,\phi)}:\mathbb{R}^{2n} \to \mathbb{R}^{n}$ such that 
 \[\eta(X_1,X_2, \Sigma) = k_{(U,\phi)}\left(\xi^{(U,\phi)}_{(X_1,X_2, \Sigma)}\circ \phi\right)\]
 on $U.$ We call $k_{(U,\phi)}$ \emph{the expression in local coordinates}. If the expressions in local coordinates can be taken so that it is independent of the choice of charts, we say the interpolation scheme is \emph{globally uniform}. The expression in local coordinates is \emph{smooth} on $D$ if $k_{(U,\phi)}$ is smooth on $D$ for all $(U,\phi).$ If the expression in local coordinates is smooth on $D$, smoothness is preserved by the application of $\eta$ as long as $D$ contains the image of the expression of $(X_1,X_2, \Sigma)$ in local coordinates.

Now we state the main theorem of this article.

\begin{theorem}[Main Theorem]\label{thm_b}
A map $\eta:PC(\mathbb{R}^n) \to Map(\mathbb{R}^n, \mathbb{R}^n)$ is an interpolation scheme of the Filippov type if and only if it satisfies the following conditions:
\begin{enumerate}
	\item $\eta$ is invariant, locally uniform and satisfies tangency condition.
	\item local expressions of $\eta$ are smooth on $D := \{({\bf p},q,{\bf r},s) \in {\mathbb R}^{2n} \mid q s < 0 \}$ and satisfies the following continuity condition: For any ${\bf p} \in \mathbb{R}^{n-1}$ and sequences $\{q_m\}, \{s_m\}$ with $q_m \to 0$, $s_m \to 0$ as $m \to \infty$, and $q_m s_m <0$, we have $\lim_{m \to \infty}k_{(U,\phi)}({\bf p},q_m,{\bf p},s_m) = ({\bf p},0).$ 
\end{enumerate}

\end{theorem}

 By the definition of the interpolation scheme, we have $k_{(U,\phi)}({\bf p},q,{\bf p}, q) = ({\bf p}, q).$ Thus the continuity property of the expression in local coordinates may be regarded as a weak notion of continuity for the interpolation scheme.

\section{Canonical form of interpolation schemes}
In this section, we prove that there is a canonical form of local expressions for locally uniform invariant interpolation schemes.

First we note that locally uniform invariant interpolation schemes are globally uniform because we assume that the discontinuity manifold $\Sigma$ is connected. Indeed, it can be checked easily that $k_{(\phi, U)} =k_{(\psi, V)}$ if $(\phi, U)$ and $(\psi, V)$ are charts of $\Sigma$ with $U\cap V \neq \emptyset.$ 

For a globally uniform interpolation scheme, we denote $\alpha_\Sigma = k_{(\phi, U)}.$

\begin{lemma}\label{lem_coin}
Let $\eta:PC(\mathbb{R}^n) \to Map(\mathbb{R}^n, \mathbb{R}^n)$ be a globally uniform interpolation scheme. If $\Sigma$ are $\Sigma'$ discontinuity manifolds such that $\Sigma \cap \Sigma' $ is an $n-1$ dimensional submanifold, then $\alpha_\Sigma = \alpha_{\Sigma'}.$
\end{lemma}
\begin{proof}
Let $(\phi, U)$ be a chart of $\Sigma \cap \Sigma'.$ Then it is a chart of both $\Sigma$ and $\Sigma'.$ By considering two piecewise-continuous vector fields $(X_1, X_2, \Sigma)$ and $(Y_1, Y_2, \Sigma')$ taking the same value at a point in $U,$ we obtain $\alpha_\Sigma =\alpha_{\Sigma'}.$
\end{proof}
Now we state and prove a theorem concerning a canonical form of local expressions for locally uniform invariant interpolation schemes.
\begin{theorem}\label{thm_canonical}
Let $\eta:PC(\mathbb{R}^n) \to Map(\mathbb{R}^n, \mathbb{R}^n)$ be a locally uniform invariant interpolation scheme. Then there is a map $\alpha: \mathbb{R}^{2n} \to \mathbb{R}^{n}$ such that, for each $x\in \Sigma,$ there exist a neighborhood $U$ of $x$ and a diffeomorphism $\Phi: \mathbb{R}^{n} \to \mathbb{R}^{n} $ such that
\[
	\eta(X_1, X_2, \Sigma)(z) = d(\Phi^{-1}) \alpha\left((d\Phi) X_1(z), (d\Phi )X_2(z)\right),
\]
for all $z \in U \cap \Sigma.$
\end{theorem}
\begin{proof}
Let $x \in \Sigma.$ Because $\Sigma$ is an $n-1$ dimensional submanifold of $\mathbb{R}^n,$ we may find a neighborhood $U_0$ of $x$ and a diffeomorphism $\phi$ such that
$\phi(U_0)$ is the open unit ball and $\phi(U_0\cap \Sigma) \subset P:=\mathbb{R}^{n-1} \times \{0\}.$ We may assume $\phi$ is orientation-preserving. Therefore, we can extend $\phi$ to a diffeomorphism $\Phi: \mathbb{R}^n \to \mathbb{R}^n,$ using Theorem 5.5 in \cite{palais1959natural}. Let $U$ be the neighborhood of $x$ where $\phi$ and $\Phi$ coincide.

Let $\alpha := \alpha_P,$ where the latter is the local expression of $\eta$ for $P.$ For each $(X_1,X_2, \Sigma)$ and $z \in U,$ we may calculate as follows.
\[
	\begin{aligned}
		\eta(X_1, X_2, \Sigma)(z) &= \eta((\Phi^{-1} \circ \Phi)_*X_1,(\Phi^{-1} \circ \Phi)_* X_2, \Phi^{-1} \circ \Phi(\Sigma))(z)\\
							&= (\Phi^{-1} )_* \eta(\Phi_*X_1,\Phi_* X_2, \Phi(\Sigma))(z)\\
							& = d(\Phi^{-1} ) \eta(\Phi_*X_1,\Phi_* X_2, \Phi(\Sigma))(\Phi(z))\\
							& = d(\Phi^{-1} ) \alpha((d \Phi) X_1(z),(d\Phi) X_2 (z)).
	\end{aligned}
\]
Here we used Lemma \ref{lem_coin}.
\end{proof}

Further, if the tangency condition is satisfied and local expressions are smooth, possible forms of the map $\alpha$ are limited.
\begin{corollary}\label{cor_rep}
Let $\eta:PC(\mathbb{R}^n) \to Map(\mathbb{R}^n, \mathbb{R}^n)$ be a locally uniform invariant interpolation scheme satisfying the tangency condition. If local expressions of $\eta$ are $C^1$ on $D := \{({\bf p},q,{\bf r},s) \in {\mathbb R}^{2n} \mid q s < 0 \}$, then the following hold for the map $\alpha$ in Theorem \ref{thm_canonical} on $D$: 

\begin{enumerate}
		\item $ \alpha({\bf p},q,{\bf r},s) = A(q,s){\bf p}+B(q,s){\bf r},$ where $A$ and $B$ are matrix-valued functions that are $0$-homogeneous with respect to $q$ and $s$.
		\item For linearly dependent vectors $({\bf p},q)$ and $({\bf r},s)$, $\alpha({\bf p},q,{\bf r},s) = {\bf 0}.$
	\end{enumerate}
\end{corollary}
\begin{proof}
Let $k>0$ be arbitrary and $A = {\rm diag}(k,k,\cdots,k,1).$
Then $A$ is regular and $F_A (P) \subset P.$ Therefore $F_A: \mathbb{R}^n \to \mathbb{R}^n$ is a diffeomorphism.

First we show that
\[k \alpha({\bf p},q,{\bf r},s) = \alpha(k{\bf p},q,k{\bf r},s)\]
for all $({\bf p},q,{\bf r},s) \in D.$ Let $(X_1, X_2, P)$ be a piecewise-continuous vector field defined by two constant vectors $({\bf p},q)$ and $({\bf r},s).$
By the definition of $\alpha$ in Theorem \ref{thm_canonical}, we have
 \[
 	\eta(X_1, X_2, P)(0) = \alpha(X_1(0), X_2(0)) = \alpha({\bf p},q, {\bf r},s).
 \]
On the other hand, the invariance of $\eta$ implies
\[
	(F_A)_*\eta(X_1, X_2, P)(0) = \eta((F_A)_*X_1, (F_A)_*X_2, A P)(0),
\]
which is
\[
	A \alpha({\bf p},q, {\bf r},s) = \alpha(k {\bf p},q, k {\bf r},s).
\]
By the tangency condition, we have $A \alpha({\bf p},q, {\bf r},s) = k \alpha({\bf p},q, {\bf r},s).$

Therefore, by Euler's homogeneous function theorem, we have $$\alpha({\bf p},q,{\bf r},s) = A({\bf p},q,{\bf r},s){\bf p} + B({\bf p},q,{\bf r},s){\bf r},$$ where $A$ and $B$ are matrix-valued functions that are $0$-homogeneous with respect to ${\bf p}, {\bf r}.$ Combined with the continuity of $A$ and $B$, we have $A({\bf p},q,{\bf r},s) = A(0,q,0,s)$ and $B({\bf p},q,{\bf r},s) = B(0,q,0,s).$ Therefore, $A$ and $B$ are functions of $q$ and $s$ only.

On the other hand, let $k>0$ be arbitrary and $A = {\rm diag}(1,1,\cdots,1,k).$ Then $A$ is regular and $F_A (P) \subset P.$ Proceeding similarly as before, we have $\alpha({\bf p},q,{\bf r},s) = \alpha({\bf p},k q, {\bf r},ks).$ Therefore, $A(q, s) = A( k q, k s)$ and $B(q, s) = B( k q, k s)$ for any $k>0.$ Thus, $A$ and $B$ are $0$-homogeneous with respect to $q$ and $s$.

Let $({\bf p},q)$ and $({\bf r},s)$ be linearly dependent. Because $q s < 0$, $({\bf p},q) = c_1({\bf k},l)$ and $({\bf r},s) = c_2({\bf k},l)$ for some ${\bf k}$, $l > 0,$ $c_1$ and $c_2$ with $c_1 c_2 < 0$. Let $M$ be a regular matrix defined by
$$
	M = \left( \begin{array}{cc}
		I_{n-1} & {\bf k} \\
		0 & l
	\end{array} \right).
$$
Then, $({\bf p},q)^T = M({\bf 0},c_1)^T$ and $({\bf r},s)^T = M({\bf 0},c_2)^T.$ Because $F_M (P) \subset P,$ we have 
\begin{eqnarray*}
M \alpha({\bf 0},c_1,{\bf 0},c_2) &=& \alpha\left(M({\bf 0},c_1)^T,M({\bf 0},c_2)^T\right) \\
						&=& \alpha({\bf p},q,{\bf r},s).
\end{eqnarray*}
 Therefore, we conclude that
\begin{eqnarray*}
	\alpha({\bf p},q,{\bf r},s) &=& M\alpha({\bf 0}, c_1, {\bf 0}, c_2)\\
						&=& M\left(A(c_1,c_2) {\bf 0} + B(c_1,c_2) {\bf 0} \right)= {\bf 0}.
\end{eqnarray*}
\end{proof}
\section{Proof of Main Theorem}\label{ch_of_fil}
In this section, we prove Main Theorem and thereby give a characterization of the Filippov convention.

The next lemma is the cornerstone of our current discussion. Essentially, it gives a characterization of the interpolation schemes of the Filippov type in cases in which the discontinuity surface is the hyperplane $P$.
\begin{lemma}\label{lem_calc}
	Let $\alpha({\bf p},q,{\bf r},s)$ be a vector-valued function from $D := \{({\bf p},q,{\bf r},s) \in {\mathbb R}^{2 n} \mid qs <0 \}$ to ${\mathbb R}^{n}$ satisfying the following conditions: 
	\begin{enumerate}
		\item $\alpha({\bf p},q,{\bf r},s) = A(q,s){\bf p}+B(q,s){\bf r},$ where $A$ and $B$ are matrix-valued functions that are $0$-homogeneous with respect to $q$ and $s$.
		\item For linearly dependent vectors $({\bf p},q)$ and $({\bf r},s)$, $\alpha({\bf p},q,{\bf r},s) = {\bf 0}.$
		\item For any ${\bf p}$ and sequences $\{q_m\}, \{s_m\}$ with $q_m \to 0$, $s_m \to 0$ as $m \to \infty$, $q_m s_m <0$, we have $\lim_{m \to \infty}\alpha({\bf p},q_m,{\bf p},s_m) = ({\bf p},0).$ 
	\end{enumerate}
Then we have \[\alpha({\bf p},q,{\bf r},s) = \left( \frac{s}{s-q}{\bf p} + \frac{q}{q-s}{\bf r}, 0 \right). \]
\end{lemma}
\begin{proof}
Fix ${\bf p} \in {\mathbb R}^{n-1}.$ By the homogeneity, we have $$\alpha({\bf p},q,{\bf r},s) = \alpha\left({\bf p},\frac{q}{m},{\bf r}, \frac{s}{m}\right)$$ for any positive $m \in {\mathbb N}.$ Therefore, we have
\begin{eqnarray*}
	(A(q,s)+B(q,s)) {\bf p}&=& \alpha({\bf p},q,{\bf p},s) \\
							&=& \lim_{m \to \infty}\alpha({\bf p},\frac{q}{m},{\bf p}, \frac{s}{m})\\
								&=&({\bf p},0)^T,
\end{eqnarray*}
for any $q$ and $s$ with $q s <0.$ Because ${\bf p}$ is arbitrary, we obtain \[A(q,s)+B(q,s) = \left(\begin{array}{c} I \\ 0 \end{array}\right).\]

Let us take ${\bf p} = q {\bf e}_i $ and ${\bf r} = s {\bf e}_i$, where ${\bf e}_i$ is the $i$-th basis vector of ${\mathbb R}^{n-1}.$ Then $({\bf p},q)$ and $({\bf r},s)$ are linearly dependent and we have
$${\bf 0} = \alpha({\bf p},q,{\bf p},s) = qA(q,s){\bf e}_i + s B(q,s) {\bf e}_i.$$ Because $i$ is arbitrary, we conclude that $qA(q,s) + s B(q,s)$ is identically $O$. Therefore, we obtain \[
A(q,s) = \frac{s}{s-q} \left(\begin{array}{c} I \\ 0 \end{array}\right),\,B(q,s) = \frac{q}{q-s} \left(\begin{array}{c} I \\ 0 \end{array}\right).\]
The conclusion follows immediately.
\end{proof}
If the map $\alpha$ in Theorem \ref{thm_canonical} has the form of the Filippov convention, we can conclude that the interpolation scheme is Filippov type.
\begin{lemma}\label{lem_fil}
Let $\eta:PC(\mathbb{R}^n) \to Map(\mathbb{R}^n, \mathbb{R}^n)$ be a locally uniform invariant interpolation scheme. If the map $\alpha$ in Theorem \ref{thm_canonical} has the form
\[\alpha({\bf p},q,{\bf r},s) = \left( \frac{s}{s-q}{\bf p} + \frac{q}{q-s}{\bf r}, 0 \right), \]
the interpolation scheme $\eta$ is Filippov type.
\end{lemma}
\begin{proof}
	Let $(X_1,X_2,\Sigma)$ be a piecewise-continuous vector field and let $x \in \Sigma$ be a point with $X_{1N}(x) X_{2N}(x)<0.$ Let $\Phi$ as in Theorem \ref{thm_canonical} and $(\mathbf{u}, q) = (d\Phi) X_1(x)$ and $(\mathbf{v}, s) = (d\Phi) X_2(x).$ Because $X_i(x) - X_{iN}(x) \mathbf{n}(x)$ is tangent to $\Sigma$ at $x$ for $i=1,2,$ we have
	\[
		0 = \mathbf{n}'(\Phi(x)) \cdot (d\Phi)_x \left(X_i(x) - X_{iN}(x) \mathbf{n}(x) \right),
	\]
for $i=1,2,$ where $\mathbf{n}'$ is the normal vector of $P.$ Therefore we have
\[
	\begin{aligned}
		q &= X_{1N}(x) (d\Phi)_x\mathbf{n}(x) \cdot \mathbf{n}'(\Phi(x))\\
		s &= X_{2N}(x) (d\Phi)_x \mathbf{n}(x) \cdot \mathbf{n}'(\Phi(x)).
	\end{aligned}
\]
We note that $(d\Phi)_x\mathbf{n}(x) \cdot \mathbf{n}'(\Phi(x)) \neq 0$ because $(d\Phi)_x$ never maps a unit normal vector to a tangent vector.
Therefore we can calculate as follows:
	\[
		\begin{aligned}
			\eta(X_1,X_2,\Sigma)(x) &= d(\Phi^{-1}) \alpha\left((d\Phi) X_1(x), (d\Phi )X_2(x)\right)\\
								& = d(\Phi^{-1}) \left( \frac{s}{s-q}{\bf u} + \frac{q}{q-s}{\bf v}, 0 \right)\\
								& = \frac{s}{s-q} d(\Phi^{-1}) \left( {\bf u} , q \right) + \frac{q}{q-s} d(\Phi^{-1}) \left( {\bf v} , s \right)\\
								& = \frac{X_{2N}(x)}{X_{2N}(x)-X_{1N}(x)} X_1(x)+ \frac{X_{1N}(x)}{X_{1N}(x)-X_{2N}(x)} X_2(x).										\end{aligned}
	\]
Therefore, the interpolation scheme $\eta$ is Filippov type.
\end{proof}
The main theorem follows using the results obtained above. 
\begin{proof}[Proof of the Main Theorem]
By Corollary \ref{cor_rep}, the map $\alpha$ satisfies the first two conditions in Lemma \ref{lem_calc}.

The third condition also holds because $\eta$ satisfies the continuity condition.

Thus, the three conditions in Lemma \ref{lem_calc} are satisfied for $\alpha.$ Therefore, $$\alpha({\bf p},q,{\bf r},s) = \left( \frac{s}{s-q}{\bf p} + \frac{q}{q-s}{\bf r}, 0 \right).$$ From Lemma \ref{lem_fil}, we conclude that the interpolation scheme $\eta$ is Filippov type.

\end{proof}

\section*{Acknowledgment}
I would like to express my gratitude to Professor Masashi Kisaka for his critique of this research. This study was supported by a Grant-in-Aid for JSPS Fellows (17J03931, 20J01101).

\bibliographystyle{plain}
\bibliography{optimality_of_filippov_}

\end{document}